\documentclass[11pt]{article}

\usepackage{amssymb,amsfonts,amsthm,amsmath,mathrsfs,mathrsfs, enumerate}
\usepackage[all]{xy}
\usepackage{geometry}
\usepackage{cleveref}
\usepackage{pst-all}
  \newcommand{\R}{\ensuremath{\mathbf{R}}}%
  \newcommand{\Z}{\ensuremath{\mathbf{Z}}}%
  \newcommand{\N}{\ensuremath{\mathbf{N}}}%
  \newcommand{\Sym}{\ensuremath{\textrm{Sym}}}%
  \newcommand{\Pf}{\mathscr{P}_\mathrm{f}}%
  \newcommand{\St}{\ensuremath{\textrm{St}}}%
  \newcommand{\supp}{\ensuremath{\textrm{supp}}}%
  \newcommand{\cati}{\ensuremath{\mathbf{I}}}%
  \newcommand{\cata}{\ensuremath{\mathbf{Amen}}}%
  \newcommand{\catac}{\ensuremath{\mathbf{AmenAct}}}%
  \DeclareMathAlphabet\mathbfcal{OMS}{cmsy}{b}{n}

  \newcommand{\G}{\ensuremath{\mathcal{G}}}%
  \renewcommand{\H}{\ensuremath{\mathcal{H}}}%

  \newcommand{\catset}{\ensuremath{\mathbf{Set}}}%

  \newcommand{\grp}{\ensuremath{\mathbf{Grp}}}%
  \newcommand{\set}{\ensuremath{\mathbf{Set}}}%
  \newcommand{\acts}{\ensuremath{\curvearrowright}}%
  \newcommand{\rk}{\operatorname{rk}_{\mathbf Q}}
  \newcommand{\iet}{\operatorname{IET}}
  \newcommand{\X}{{X}}

  \newcommand{\sym}{\operatorname{Sym}}

\theoremstyle{definition}
  \newtheorem{defin}{Definition}[section]

  \newtheorem{question}[defin]{Question}

\theoremstyle{plain}
  \newtheorem{thm}[defin]{Theorem}
  \newtheorem{main thm}{Theorem}
  \newtheorem{prop}[defin]{Proposition}
  
  \newtheorem{cor}[defin]{Corollary}
  \newtheorem{lemma}[defin]{Lemma}

\theoremstyle{remark}
  \newtheorem{remark}[defin]{Remark}
  \newtheorem{remarks}[defin]{Remarks}
  \newtheorem{example}[defin]{Example}
  
\begin{document}

\author{Kate Juschenko \and Nicol\'as Matte Bon \and Nicolas Monod \and Mikael de la Salle}
\title{Extensive amenability\\ and an application to interval exchanges}
\date{January 2015}

\maketitle

\begin{abstract}
Extensive amenability is a property of group actions which has recently been used as a tool to prove amenability of groups. We study this property and prove that it is preserved under a very general construction of semidirect products. As an application, we establish the amenability of all subgroups of the group $\iet$ of interval exchange transformations that have angular components of rational rank~${\leq 2}$.

In addition, we obtain a reformulation of extensive amenability in terms of inverted orbits and use it to present a purely probabilistic proof that recurrent actions are extensively amenable. Finally, we study the triviality of the Poisson boundary for random walks on $\iet$ and show that there are subgroups $G <\iet$ admitting no finitely supported measure with trivial boundary.
\end{abstract}

\section{Introduction}
Recall that an action of a group $G$ on a set $X$ is \emph{amenable} if it admits a $G$-invariant mean, i.e.\ a finitely additive probability measure defined on all subsets of $X$ that is invariant under the action of $G$. 

In this article, we study \emph{extensive amenability} of an action $G\acts X$. As the name is intended to suggest, this is a much stronger form of amenability and it has an intimate connection with extensions of groups and of actions. This property was introduced (without a name) in~\cite{JM} as a tool to prove amenability of groups, a role it continued to play in~\cite{JNS}. In order to give its formal definition, we denote by $\Pf(X)$ the set of all finite subsets of $X$.

\begin{defin}\label{condi=fullweight}
The action of a group $G$ on a set $X$ is \emph{extensively amenable} if there is a $G$-invariant mean on $\Pf(X)$ giving full weight to the collection of subsets that contain any given element of $\Pf(X)$.
\end{defin}

\begin{remarks}\ 
\begin{enumerate}[(i)]
\item Extensively amenable actions are amenable unless $X=\varnothing$, and every action of an amenable group is extensively amenable; see Lemma~\ref{lemma=intermediatepty}.
\item Both converses fail: there are amenable actions that are not  extensively amenable (see~\cite{JS}, or \S~\ref{section:frankenstein} for a new example), as well as extensively amenable faithful actions of nonamenable groups~\cite{JM}.
\item A sufficient condition for extensive amenability was provided in~\cite{JNS}, namely \emph{recurrence} in the sense of random walks; see \S~\ref{S: recurrent} for details and a new proof of this result.
\item We do not consider any topology on $G$ or $X$. Moreover, it is not relevant to assume $G$ or $X$ countable, because an action $G\acts X$ is extensively amenable if and only if the action of every finitely generated subgroup $H\le G$ on every $H$-orbit is so (Lemma~\ref{lemma=JSbis}).
\end{enumerate}
\end{remarks}

An important advantage of extensive amenability over usual amenability of actions is that it is more robust. For instance, it is preserved under extensions of actions in a sense made precise in \S~\ref{section:elementary}. Another stability property is fundamental for applications and concerns a wide class of group actions constructed functorially from $G\acts X$; here the group $G$ itself will be replaced by a semidirect product extension of $G$. The first example is as follows:

Given $G\acts X$ one can form the \emph{permutational wreath product}, or ``lamplighter group'', $(\Z/2\Z)^{(X)}\rtimes G$. Here $(\Z/2\Z)^{(X)}$ denotes the group of finitely supported configurations $X\to \Z/2\Z$, on which $G$ acts by translations; it can be identified with $\Pf(X)$. The lamplighter group acts affinely on the ``lamp group'' $(\Z/2\Z)^{(X)}$, identified with the coset space $((\Z/2\Z)^{(X)}\rtimes G)/G$. By simple Fourier analysis arguments, it was shown in~\cite[3.1]{JM} that the amenability of this affine action is equivalent to the extensive amenability of $G\acts X$. 

Our first result (Theorem~\ref{prop:functor:intro} below) provides a vast non-commutative generalization of this equivalence, wherein the lamp groups $(\Z/2\Z)^{(X)}$ will be replaced by rather arbitrary functors $F$ applied to $X$. Specifically, we consider all "finitary" functors to the category of amenable groups in the following precise sense.

Denote by $\cati$ the category of finite sets whose morphisms are injective maps. Denote further by $\cata$ the category of amenable groups with group homomorphisms. Since $\cata$ has direct limits, any functor $F\colon\cati\to\cata$ extends to the category of all sets with injective maps as morphisms by setting $F(X)$ to be the direct limit of $F(Y)$ as $Y$ runs over the directed set $\Pf(X)$. We still denote the resulting functor by the same letter $F$. Several explicit examples are given in \S~\ref{section:functors}, but the reader can already have in mind $F(X) = A^{(X)}$ for an amenable group $A$, or $F(X) = \Sym(X)$ the group of permutations of $X$ with finite support.

Notice that any group action $G\acts X$ yields an action $G\acts F(X)$ by automorphisms and hence a semidirect product group $F(X)\rtimes G$. This semidirect product acts on $F(X)$ identified with the coset space $(F(X)\rtimes G)/G$. 
 
We say that the functor $F$ is \emph{tight} on a nonempty set $X$ if for all (equivalently, some) $x\in\X$ the morphism $F(X\setminus \{x\})\to F(X)$ induced by inclusion is not onto. A particular case of our results from \S~\ref{section:functors} is the following theorem.

\begin{thm}\label{prop:functor:intro}
Let $F\colon\cati\to\cata$ be any functor, extended to arbitrary sets as described above. Let $G$ be a group acting on a set $X$.

If the action $G\acts X$ is extensively amenable, then the action $F(X) \rtimes G\acts F(X)$ is amenable. 
Moreover it is extensively amenable.

Conversely, assume that the action $F(X) \rtimes G\acts F(X)$ is amenable. Then $G\acts X$ is extensively amenable provided $F$ is tight on $X$.
\end{thm}

In particular, the amenability of the action $F(X) \rtimes G\acts F(X)$ does not depend on the choice of the functor $F$ provided that it is tight on $X$. Furthermore, for such ``affine'' actions, amenability and extensive amenability coincide.

Beside the lamplighter case $F(X)=(\Z/2\Z)^{(X)}$, another concrete example to which Theorem~\ref{prop:functor:intro} applies is $F(X)=\sym(X)$, the group of finitely supported permutations of a set $X$, on which $G$ acts by conjugation. See \S~\ref{section:functors} for more examples. A case in which the theorem \emph{does not} apply, as it does not come from a functor on $\cati$, is $F(X)=(\Z/2\Z)^X$, the \emph{unrestricted} direct product indexed by $X$; see \S~\ref{section:unrestricted} for a counterexample in this case.

\medskip
Theorem~\ref{prop:functor:intro} can be used to establish a criterion for the amenability of some subgroups of $F(X) \rtimes G$:

\begin{cor}\label{prop:functor:amenable:intro}
Let $G \acts X$ be an extensively amenable action and let $F\colon \cati \to \cata$ be any functor.

A subgroup $H$ of $F(X) \rtimes G$ is amenable as soon as the intersection $H \cap (\{1\} \times G)$ is so.
\end{cor}

(The derivation of this result from Theorem~\ref{prop:functor:intro} is given in \S~\ref{section:functors}.)

%

\begin{remark}\label{R: method}
A particular case in which this criterion applies is when one is able to construct a \emph{twisted embedding} $G\hookrightarrow F(X)\ltimes G$ of the form $g\mapsto (c_g, g)$ with the property that $\{g\in G\: : \: c_g=1\}$ is an amenable subgroup of $G$. We then say that  $c\colon G\to F(X),\  g \mapsto c_g $ is a \emph{$F(X)$-cocycle with amenable kernel}. The conclusion is then that $G$ is amenable.
\end{remark}

This method was used in the proof that the topological full group of any minimal Cantor system is amenable~\cite{JM}; namely the particular case of Corollary~\ref{prop:functor:amenable:intro} for the ``lamp'' functor $F(X)= (\Z/2\Z)^{(X)}$ was used. In the next paragraph we present a new application of this method, relying this time on the functor $F(X)=\sym(X)$.

\paragraph{Application to groups of interval exchange transformations.}

An interval exchange transformation is a permutation of a circle obtained by cutting this circle into finitely many intervals (arcs), and reordering them. More precisely, an interval exchange transformation is a right-continuous permutation $g$ of $\R/\Z$ such that the set $\{gx - x, x \in \R/\Z\}$, called the set of angles of $g$, is finite. Interval exchange transformations are traditionally defined as permutations of an interval $[a,b)$, but we get the same notion by identifying $a$ and $b$.

The set of all interval exchange transformations is a group acting on $\R/\Z$, denoted by $\iet$. 

The interval exchange transformations have been a popular object of study in dynamical systems and ergodic theory; for example they model the dynamics on polygonal billiards with rational angles. The exchanges of three or more intervals were first considered by Katok and Stepin~\cite{KS}. The systematic study started in the paper by Keane~\cite{keane}, who introduced the terminology. For an introduction and account of the results, see the survey by Viana
\cite{viana} or the book of Katok and Hasselblatt ~\cite{KH}.

A basic question on $\iet$ was raised by  Katok, namely whether or not $\iet$ contains a non-abelian free group. This problem attracted some attention recently; for instance, Dahmani--Fujiwara--Guirardel~\cite{DFG} showed that free subgroups in $\iet$ are {\it rare} in the sense that a group generated by a generic pair of interval exchange transformations is not free. A related open question raised in~\cite[p.4]{Cornulier:Bourbaki} is as follows.

\begin{question}\label{Q: iet amenable}
Is the group $\iet$ amenable?
\end{question}

It is sufficient for these questions to consider finitely generated subgroups of $\iet$. The latter can be classified according to their \emph{rational rank} as follows. Given a subgroup $G \le \iet$, we denote by $\Lambda(G)\le \R/\Z$ its \emph{group of angles}, i.e. the subgroup of $\R/\Z$ generated by all increments $gx-x$ where $g\in G$ and $x\in \R/\Z$.

\begin{defin}\label{def=rational_rank_iet}
The \emph{rational rank} of $G$, denoted $\rk(G)\in \N\cup\{\infty\}$, is the supremum of all $d$ such that $\Z^d$ embeds in $\Lambda(G)$.
\end{defin}

\begin{remark}
If $G$ is finitely generated, then $\Lambda(G)$ is a finitely generated abelian group and thus $\Lambda(G)\simeq \Z^d\times H$ with $H$ a finite abelian group. In this case $\rk(G)=d$.
\end{remark}

\begin{thm}\label{T: rank2intro}\label{T:rank2}
Let $G \le \iet$. If $\rk(G)\leq 2$ then $G$ is amenable. 
\end{thm}

The proof is based on the method outlined in Remark~\ref{R: method} with $F(X)=\sym(X)$, the group of finitely supported permutations of $X$. We first observe that there exists a $\sym(\R/\Z)$-cocycle $g\mapsto \tau_g$ with amenable kernel defined on the whole group $\iet$. This reduces the problem of the amenability of a subgroup $G\le\iet$ to proving the extensive amenability of the action $G\acts \R/\Z$. In the case $\rk(G)\leq 2$, this action is recurrent and this yields Theorem~\ref{T: rank2intro}. The general case would follow from a positive answer to Question~\ref{Q: polynomial} below.

The possibility to define a twisted embedding $\iet \hookrightarrow \sym (\R/\Z) \rtimes \iet$  also has an almost immediate  application related to \emph{Poisson--Furstenberg boundaries} for random walks on $\iet$.   A group endowed with a probability measure $(G, \mu)$ is said to have the \emph{Liouville property} if its Poisson--Furstenberg boundary is trivial. If the support of $\mu$ generates $G$, the Liouville property for $(G, \mu)$ implies amenability of $G$~\cite[\S 4.2]{Kaimanovich-Vershik}. Conversely any amenable group admits a symmetric measure supported on a 
(possibly infinite) generating set so that $(G, \mu)$ has the Liouville property~\cite[Theorem 4.3]{Kaimanovich-Vershik}. The next result shows that there is an obstruction to showing amenability of $\iet$ in this way, namely that the Liouville property never holds for finitely supported measures generating a ``large enough'' subgroup. 

We say that a probability measure on a group is \emph{non-degenerate} if its support generates the group.

\begin{thm} \label{T:rank3} Let $G<\iet$ be finitely generated.
\begin{enumerate}[(i)]
\item If $\rk(G)=1$ then every symmetric, finitely supported probability measure $\mu$ on $G$ has the Liouville property.
\item Assume that $\rk(G)\geq 3$ and, viewing $\Lambda(G)$ as a subgroup of $\iet$ consisting of rotations, assume that $G$ contains $\Lambda(G)$ strictly. Then every finitely supported, non-degenerate probability measure on $G$ has non-trivial Poisson--Furstenberg boundary.
\end{enumerate}
\end{thm}

Part~(i) follows essentially from a combination of known facts, namely a result from~\cite{Matte} combined with the observation~\cite{Cornulier:Bourbaki} that a finitely generated $G<\iet$ can be realized as a subgroup of the \emph{topological full group} of a certain minimal subshift on the group of angles $\Lambda(G)$. Part~(ii) is proven applying the twisted embedding $\iet\hookrightarrow \sym(X)\rtimes \iet$ and then arguing in a similar way as for the classical example of the lamplighter group $(\Z/2\Z)\wr \Z^3$ from~\cite{Kaimanovich-Vershik}. Details are given in \S~\ref{S: Liouville}. \medskip

There is a strong analogy between Theorems~\ref{T: rank2intro} and~\ref{T:rank3} and existing results on groups generated by finite \emph{automata of polynomial activity}. These are a well-studied class of groups acting on rooted trees, which can be classified according to their \emph{activity degree} $d\in \N$  (we refer to~\cite{Sidki} for  details). Amenability of polynomial activity automata groups is an open problem, that has been answered affirmatively for $d\leq 2$~\cite{JNS}.  Concerning the Liouville property, it was conjectured in~\cite{Amir-Angel-Virag:linearautomata} that it holds up to $d=2$; this is known for $d=0$~\cite{Bartholdi-Kaimanovich-Nekrashevych:boundedautomata,Amir-Angel-Matte-Virag} and for $d=1$ in some important special cases~\cite{Amir-Angel-Virag:linearautomata}, while for $d\geq 3$ the Liouville property does not hold~\cite{Amir-Virag:positivespeed}. 

\paragraph{Criteria for extensive amenability.}
It is an intriguing problem to find new criteria to establish extensive amenability of an action $G\acts X$. To this end, there is no loss of generality in assuming that $G$ is finitely generated and acts transitively on $X$ (Lemma~\ref{lemma=JSbis}). Assuming this, let $S$ be a finite symmetric generating set of $G$. Recall that the \emph{orbital Schreier graph} of the action is the oriented labelled graph $\Gamma(G, X, S)$ whose vertex set is $X$ and edge set is $X \times S$, where the edge $(x,s)$ connects $x$ to $sx$. The corresponding edge is labelled by $s$.

We would like to find sufficient geometric conditions on the labelled Schreier graph $\Gamma(G, S, X)$ that imply extensive amenability of the action $G\acts X$. An obvious simplification for this problem is to seek sufficient conditions that only depend on the \emph{unlabelled} Schreier graph. So far the only available criterion~\cite{JNS} of this kind is Theorem~\ref{T: recurrent}.

\begin{question}\label{Q: polynomial}
Assume that $\Gamma(G, X, S)$ grows polynomially. Does this imply that $G\acts X$ is extensively amenable?
\end{question} 
A positive answer to Question~\ref{Q: polynomial} would imply that the whole group $\iet$ is amenable, as the Schreier graph of any finitely generated subgroup $G<\iet$ admit an injective Lipschitz embedding into $\Z^d$ (see \S~\ref{section:iet}). Note that a counterexample is missing even for graphs with uniform sub-exponential growth. Relying on the main result from \cite{JNS}, a positive answer in this case would also imply amenability of polynomial activity automata groups, for which the graphs grow uniformly sub-exponentially \cite{Bondarenko}.

\paragraph{Organization of the paper.} In \S~\ref{section:elementary} we study general properties of extensive amenability. \S~\ref{section:functors} deals with the functorial formalism, and contains proofs and generalizations of Theorem~\ref{prop:functor:intro} and Corollary~\ref{prop:functor:amenable:intro}. In \S~\ref{S: inverted orbit} we give a probabilistic reformulation of extensive amenability. The results on groups of interval exchange transformations are proven in \S~\ref{section:iet}. In \S~\ref{section:frankenstein} we describe an example of an action $G\acts X$ which is not extensively amenable, but the action of every subgroup $H$ of $G$ on every $H$-orbit is amenable (see Corollary~\ref{C: hereditarily amenable}). Finally in \S~\ref{section:unrestricted} we construct an example where the conclusion of Theorem~\ref{prop:functor:intro} does not hold for the unrestricted wreath product action $(\Z/2\Z)^X \rtimes G \acts (\Z/2\Z)^X$.

\paragraph{Acknowledgements.}
We are grateful to Gidi Amir for many conversations related to \S~\ref{S: inverted orbit}, to Anna Erschler for pointing out a reference for Lemma~\ref{L: non symmetric transient}, and to Yves de Cornulier for useful and numerous comments on the exposition.
\section{General properties of extensive amenability}\label{section:elementary}

\paragraph{Notation.} We will always identify a finitely additive probability measure $m$ defined on all subsets of $X$ with the corresponding positive unital linear form on $\ell^\infty(X)$, that we denote as $f \mapsto m(f)$ or $f \mapsto \int f dm$ or $\int_X f(x) dm(x)$.

\bigskip
The first Lemma is easy and already known from~\cite{JS}. We give a proof for convenience.
\begin{lemma}\label{lemma=intermediatepty} Every action of an amenable group is extensively amenable, and every extensively amenable action on a nonempty set is amenable.

\end{lemma}

\begin{proof}
Assume that $G$ is amenable and acts on a set $X$. Then $G$ acts on the set $K$ of all means on $\Pf(X)$ giving full weight to any given finite subset of $X$. Observe that $K$ is a $\sigma(\ell^\infty(X)^*,\ell^\infty(X))$-closed (hence compact) convex subset of $\ell^\infty(X)^*$. Moreover $K$ is nonempty because it contains any cluster point of the net $(\delta_{A})_{A \in \Pf(A)}$. Since $G$ is amenable, $G$ fixes a point in $K$, which is exactly extensive amenability.

If the action of a group $G$ on a nonempty set $X$ is extensively amenable in particular there is a $G$-invariant mean $m$ on $\Pf(X) \setminus \{\varnothing\}$. Then the mean $f \in \ell^\infty(X) \mapsto \int_{\Pf(X)\setminus  \{\varnothing\}} \frac{1}{|A|}\sum_{x \in A} f(x) dm(A)$ is $G$-invariant.
\end{proof}

In the next lemma we study permanence properties of extensive amenability.
The equivalence between~(i) and~(ii) in particular shows that extensive amenability is preserved by direct limits.

\begin{lemma}\label{lemma=JSbis} Let $G$ be a group acting on a set $X$. The following are equivalent:
\begin{enumerate}[(i)]
\item\label{item1} The action of $G$ on $X$ is extensively amenable.
\item\label{item2} For every finitely generated subgroup $H$ of $G$ and every $H$-orbit $Y \subset X$, the action of $H$ on $Y$ is extensively amenable.
\item\label{item3} For every finitely generated subgroup $H$ of $G$ and every $x_0 \in X$, there is an $H$-invariant mean on $\Pf(H x_0)$ that gives nonzero weight to $\{A \in \Pf(Hx_0), x_0 \in A\}$.
\item\label{item4} There is a $G$-invariant mean on $\Pf(X)$ that gives nonzero weight to $\{A \in \Pf(X), x_0 \in  A\}$ for all $x_0 \in X$.
\end{enumerate}
\end{lemma}

\begin{proof}~\eqref{item1}$\Rightarrow$\eqref{item4} holds by definition.

\eqref{item4}$\Rightarrow$\eqref{item3}. Denote by $Y=Hx_0$. The map $f \in \ell^\infty(\Pf(Y)) \to f(\cdot \cap Y) \in \ell^\infty(\Pf(X))$ is a positive unital $H$-map, so that the composition of the mean given by~\eqref{item4} is a $H$-invariant mean giving positive weight to the sets containing $x_0$.

\eqref{item3}$\Rightarrow$\eqref{item2} follows from~\cite[Lemma 3.1]{JM}.

\eqref{item2}$\Rightarrow$\eqref{item1}. For every finitely generated subgroup $H$ of $G$ and every finite union $Y=Y_1\cup\dots \cup Y_n$ of $H$-orbits, we have means $m_i$ on $\Pf(Y_i)$ given by~\eqref{item2}, and we construct a mean $m_{H,Y}$ on $\Pf(X)$ as follows:
\[ \int f d m_{H,Y} = \int_{\Pf(Y_1)} \dots \int_{\Pf(Y_n)} f(B_1\cup\dots \cup B_n) d m_1(B_1)\dots d m_n(B_n).\]
The mean $m_{H,Y}$ is $H$-invariant and gives full weight to the sets containing any given finite subset of $Y$. If we order the pairs $(H,Y)$ by inclusion, any cluster point of the net $(m_{H,Y})$ is $G$-invariant and gives full weight to the sets containing any given finite subset of $X$.\end{proof}

We say that a group action $G\acts X$ is  \emph{hereditarily amenable} if for every subgroup $H$ of $G$ the action of $H$ on every $H$-orbit is amenable. Thereby we have
\begin{cor}\label{C: hereditarily amenable}
Extensively amenable actions are hereditarily amenable.
\end{cor}
In \S~\ref{section:frankenstein} we shall see that the converse is not true.

\bigskip
We now prove that in some sense extensive amenability is preserved by extensions of actions. In the statement below, we denote by $G_y$ the stabilizer of $y$ in $G$.

\begin{prop} \label{P: extension} Let $G$ be a group acting on two sets $X,Y$ and let $q\colon X \to Y$ be a $G$-map. If $G \acts Y$ is extensively amenable and if $G_y \acts q^{-1}(y)$ is extensively amenable for every $y \in Y$, then $G \acts X$ is extensively amenable. The converse holds if $q$ is surjective.
\end{prop}

The special case of transitive actions can be reformulated as the following particularly clean equivalence. We recall that the corresponding statement for amenability \emph{does not hold}, see~\cite{Monod-Popa}.

\begin{cor}
Let $F \le H \le G$ be groups. Then the action of $G$ on $G/F$ is extensively amenable if and only if both the actions of $H$ on $H/F$  and $G$ on $G/H$ are extensively amenable.
\end{cor}

\begin{proof}[Proof of the corollary]
Apply Proposition~\ref{P: extension} to $X=G/F$ and $Y = G/H$ with the quotient map.
\end{proof}

\begin{proof}[Proof of Proposition~\ref{P: extension}]
Assume that $G \acts Y$ and $G_y \acts q^{-1}(y)$ are extensively amenable for all $y \in Y$. Take a mean $m_1$ on $\Pf(Y)$ as in Definition~\ref{condi=fullweight} for $G \acts Y$. We claim that for all $y \in Y$ there is a mean $m_y$ on $\Pf(q^{-1}(y))$  giving full weight to the subsets of $q^{-1}(y)$ that contain any given finite subset of $q^{-1}(y)$ and with the property that $m_{gy}$ is the push-forward by $g$ of the mean $m_y$ for all $y \in Y$ and $g \in G$. Indeed, by the assumption that $G_y \acts q^{-1}(Y)$ is extensively amenable we can take such a $G_y$-invariant mean $m_y$ on $\Pf(q^{-1}(y))$ for each $y$ in a fixed $G$-transversal of $Y$, and, for $y'=gy$ in the $G$-orbit of such $y$, define $m_{y'}$ as the push-forward of $m_y$ by $g$; this definition does not depend on $g$ because $m_y$ is $G_y$-invariant, and this defines the requested mean on $\Pf(q^{-1}(y'))$.

For any subset $A =\{y_1,\dots,y_n\} \in \Pf(Y)$, denote by $m_A$ the mean on $\Pf(X)$ by
\[ m_A(f) = \frac{1}{n!}\sum_{\sigma\colon \{1,\dots,n\} \to A} \int_{q^{-1}(\sigma(1))} \dots \int_{q^{-1}(\sigma(n))} f(\cup_{i=1}^n B_{i}) dm_{\sigma(1)}(B_1)\dots dm_{\sigma(n)}(B_{n}).\]
The average is taken over all bijections $\sigma$, in order to ensure that $m_A$ does not depend on the chosen ordering of the elements of $A$. By the properties of the means $m_y$, we have the following two properties of the means $m_A$. Firstly the push-forward of $m_A$ by $g \in G$ is $m_{gA}$; secondly $m_A$ gives full weight to the subsets of $X$ that contain any given subset of $q^{-1}(A)$. These properties ensure that the mean $m$ on $\mathcal P_f(X)$ defined by
\[ m \colon f \in \ell^\infty(\Pf(X)) = \int_{\mathcal P_f(Y)} m_A(f) dm_1(A),\]
is $G$-invariant and gives full weight to the subsets of $X$ that contain any given subset of $X$. This proves that $G \acts X$ is extensively amenable.

For the converse, assume first that $G \acts X$ is extensively amenable, and take $m$ a mean on $\Pf(X)$ as in Definition~\ref{condi=fullweight}. Then by Lemma~\ref{lemma=JSbis} the action of the subgroup $G_y$ on the subset $q^{-1}(y)$ of $X$ is extensively amenable for all $y \in Y$. Consider $q'\colon A \in \Pf(X) \mapsto \{q(a),a \in A\} \in \Pf(Y)$. The map 
\[ f \in \ell^\infty(\Pf(Y)) \mapsto m( f \circ q')\]
is a $G$-invariant mean that gives full weight to the sets containing every finite subset of $q(Y)$. This proves that $G \acts q(Y)$ is extensively amenable.
\end{proof}

\begin{cor}\label{lemma=wreath product} Let $G$ and $H$ be two groups, and $G\acts X$ and $H\acts Y$ be two extensively amenable actions. Then
\begin{enumerate}[(i)]
\item the diagonal action of $G \times H$ on $X \times Y$ is extensively amenable.
\item the action on $H^X \rtimes G$ on $X \times Y$ is extensively amenable.
\end{enumerate}
\end{cor}
In the second statement, $G$ acts diagonally on $X \times Y$ (with trivial action on $Y$), and $H^{X}$ acts by $(h_x)_{x \in X} \cdot (x,y) = (x,h_x y)$, and this gives rise to an action of $H^X \rtimes G$. 
\begin{proof} The first statement is Proposition~\ref{P: extension} for the actions of $G \times H$ on $X \times Y$ and $X$ (trivial action of $H$), and $q \colon X \times Y \to X$ the first coordinate projection.

The second statement is Proposition~\ref{P: extension} for the actions on $H^X \rtimes G$ on $X \times Y$ and $X$ (trivial action of $H^X$) for the same projection $q$.
\end{proof}

\section{Functors from sets to amenable group actions}\label{section:functors}

This section deals with the proof (and generalizations) of Theorem~\ref{prop:functor:intro} and~\ref{prop:functor:amenable:intro}. We start by giving examples of functors $\cati \to \cata$. 

\begin{example}\label{ex:wreath}
Fix an amenable group $A$. Consider the functor $F$ which maps any finite set $Y$ to $A^Y$ with the obvious extension map on inclusions. Then for a general set $X$ we have $F(X)=A^{(X)}$, the restricted product.
\end{example}

\begin{example}\label{ex:sym}
If $F(Y)$ is the symmetric group of the finite set $Y$, then $F(X)$ will be the finitely supported permutation group $\Sym(X)$ of a general set $X$.
\end{example}

In the following two examples rings are always assumed to have a unit, but are not assumed to be commutative.

\begin{example}\label{ex:matrix}
Fix a finite ring $R$ (rings are always assumed to have a unit). Consider the functor $F(Y)$ given by the group of invertible matrices over $R$ indexed by a finite set $Y$, with the ``corner'' inclusions. For an infinite set $X$, the group $F(X)$ is a stable linear group which we denote by $GL_{(X)}(R)$ (it is the usual one when $X$ is countable). As a variation, we can define $F(Y)$ to be the group $EL_Y(R)$ generated by elementary matrices; this yields stable elementary groups $EL_X(R)$ for arbitrary sets $X$.
\end{example}

\begin{example}\label{ex:Steinberg}
Fix again a finite ring $R$. Given a finite set $Y$ consider the (unstable) Steinberg group or degree $|Y|$ over $R$. Then for $X$ general we obtain a (stable) Steinberg group which we denote by $\St_{(X)}(R)$. Thus we have a natural transformation given by the morphisms $\St_{(X)}(R)\to EL_X(R)$ (recall that these are isomorphisms if $R$ is a finite field).
\end{example}

\begin{example}\label{ex:trivial}
Finally, a trivial example is given by the constant functor $F(Y)=A$ for any given amenable group $A$. We also have the modified constant functor $F_0$ defined by $F_0(Y)=A$ for all non-empty sets $Y$ and $F_0(\varnothing)$ being the trivial group.
\end{example}

\medskip

Recall the definition of a tight functor from the introduction.

\begin{defin}\label{defin:tight}
A functor $F\colon\cati\to\cata$ is called \emph{tight} on a (possibly infinite) set $X$ if for all $x\in X$ the morphism $F(X\setminus \{x\}) \to F(X)$ is not onto.
\end{defin}

It is straightforward to verify that all examples above are tight on every set $X$ except some degenerate cases: Example~\ref{ex:trivial}, Example~\ref{ex:wreath} with $A$ trivial and $X$ nonempty, and in Examples~\ref{ex:sym},~\ref{ex:matrix} and~\ref{ex:Steinberg} one should exclude $|X|= 1$.

\medskip

We first treat separately a special case of Theorem~\ref{prop:functor:intro}, which is enough for the application to interval exchange transformations. The proof of the full statement of Theorem~\ref{prop:functor:intro} will be given later in a more general setting, see Theorem~\ref{prop:functor:modern}.
\begin{prop}\label{prop:functor:cata}
Let $F\colon\cati\to\cata$ be a functor as in Theorem~\ref{prop:functor:intro}, and assume moreover that $F(A)$ is a finite group whenever $A$ is a finite set.

If the action of $G$ on $X$ is extensively amenable, then the action of $F(X) \rtimes G$ on $F(X)$ is extensively amenable.

\end{prop}

\begin{proof}
Assume that the action of $G$ on $X$ is extensively amenable. Let $m$ be a $G$-invariant mean on $\Pf(X)$ giving full weight to the collection of subsets that contain any given finite subset. Let $F$ be a functor from $\cati$ to the category of finite groups.

We first prove that the action of $F(X) \rtimes G$ on $F(X)$ is amenable. Then we will see how to adapt the proof to show that the action is extensively amenable.

For a finite set $A$, let $m_A$ be the uniform probability measure on $F(A)$, which is a finite group by assumption.
We denote by $m_A^X$ the mean on $F(X)$ obtained by push-forward through $F(A)\to F(X)$; this mean is $F(A)$-invariant by construction. Observe also that for we have $gm_A^X=m_{gA}^X$ for every $g\in G$.

We obtain a mean $\widetilde m$ on $F(X)$ by integrating $m_A^X$ over $m$; more precisely, given $f\in \ell^\infty(F(X))$ we define
$$\widetilde m(f) = \int_{\Pf(X)} m_A^X(f) dm(A).$$
This mean is $G$-invariant by construction and we claim that it is also $F(X)$-invariant. It is enough to show that $m$ is $F(A)$-invariant for every finite subset $A$ of $X$. But this holds because $m$ gives full weight to the set of finite subsets containing $A$, and since $m_B^X$ is $F(A)$-invariant whenever $B$ contains $A$. 

To prove that the (transitive) action of $F(X) \rtimes G$ on $F(X)$ is extensively amenable, by~\cite[Lemma 3.1]{JM} we have to prove that there is a mean on $(\Z/2\Z)^{(F(X))}$ that is invariant under the action of $(\Z/2\Z)^{(F(X))}\rtimes(F(X) \rtimes G) $. For every $A \in \mathcal P_f(X)$, the map $F(\subset) \colon F(A) \to F(X)$ induces a group homomorphism $(\Z/2\Z)^{F(A)} \to (Z/2\Z)^{(F(X))}$ obtained by sending $\delta_s$ to $\delta_{F(\subset)(x)}$. If $m_A$ is the uniform probability measure on the (finite) image of this group homomorphism, then the same argument  as above shows the mean obtained on $(\Z/2\Z)^{(F(X))}$ by integrating $m_A$ over $m$ is invariant under $(\Z/2\Z)^{(F(X))}\rtimes (F(X) \rtimes G)$.
\end{proof}
We now proceed to the 
\begin{proof}[Proof of Corollary~\ref{prop:functor:amenable:intro}] To prove that $H$ is amenable we find an amenable action of $H$ with amenable stabilizers. Let $e$ be the unit element of $F(X)$, and consider the action of $H$ on the $H$-orbit of $e$. By Theorem~\ref{prop:functor:intro}  and Lemma~\ref{lemma=JSbis}, this action is extensively amenable and in particular is amenable. Moreover the stabilizer of $e$ is $H\cap (\{1\} \times G)$ which by assumption is amenable. So is every other stabilizer, which is a conjugate of the stabilizer of $e$. This implies that $H$ is amenable.
\end{proof}
\begin{remark}
In the above proof we appealed to Theorem~\ref{prop:functor:intro} in full generality, however it is sufficient to apply Proposition~\ref{prop:functor:cata} if its assumptions are satisfied.
\end{remark}
\medskip
%

%
%

To prove the full statement of  Theorem~\ref{prop:functor:intro} it is convenient to pass to a slightly more general setting of functors from $\cati$ to the category $\catac$ of amenable actions. This generalization is needed in order to obtain that the action of $F(X)\ltimes G$ on $F(X)$ is not only amenable, but also extensively amenable.

The objects in $\catac$ are given by a group acting amenably on a set, and a morphism from $G \acts X$ to $H \acts Y$ is given by a map $X \to Y$ and a group homomorphism $G \to H$ that intertwines the two actions. If $F\colon \cati\to\catac$ is a functor, we denote by $F_\grp$ and $F_\set$ the associated functors from $\cati$ to the category of groups and sets respectively. That is, $F(X)$ is the action of $F_{\grp}(X)$ on $F_{\set}(X)$. Since $\catac$ has direct limits, any functor $F\colon\cati\to\catac$ extends to the category of all sets with injective maps as morphisms, in the same way as for functors to $\cata$. We still denote the resulting functor by $F$.

As before, notice that any group action $G\acts X$ yields an action $G\acts F_\grp(X)$ by automorphisms and an action $G \acts F_\set(X)$, hence an action $F_\grp(X)\rtimes G \acts F_\set(X)$.

\begin{example}\label{ex:groups} Any functor $F \colon \cati \to \cata$ is in particular a functor $\cati \to \catac$ ($G \mapsto (G \acts G)$ realizes $\cata$ as a subcategory of $\catac$).
\end{example}

\begin{example}\label{ex:composition} Take two functors $E,F\colon \cati \to \catac$. Assume furthermore that $F_\set$ is a functor from $\cati$ to the category of sets with injective maps as morphisms, and $F_\grp$ is a functor $\cati \to \cata$. For a finite set $Y$, we have an action of an amenable group $F_\grp(Y)$ on $F_\set(Y)$, and hence by the functoriality of $E$ an action of $E_\grp(F_\set(Y)) \rtimes F_\grp(Y)$ on $E_\set(F_\set(Y))$. This action is amenable because $E_\grp(F_\set(Y)) \acts E_\set(F_\set(Y))$ is amenable and the group $F_\grp(Y)$ is amenable. This defines a new functor, which we denote $E \diamond F\colon \cati \to \catac$.
\end{example}

\begin{example}\label{ex:composition2} The previous construction also makes sense if $F\colon \cati \to \catac$ is a functor such that $F_{\grp}\colon \cati \to \cata$ and $E\colon \catset \to \catac$.
\end{example}
\medskip

\begin{defin}
A functor $F\colon\cati\to\catac$ is called \emph{tight} on a (possibly infinite) set $X$ if for all $x\in X$, no invariant mean for $F(X)$ gives full weight to the image of $F_\set(X\setminus \{x\}) \to F_\set(X)$.
\end{defin}
This definition extends Definition~\ref{defin:tight} to functors $F\colon \cati \to \catac$, because if $H$ is a subgroup of an amenable group $G$ then any $G$-invariant mean on $G$ gives weight $|G\colon H|^{-1}$ to $H$, where $|G\colon H|$ is the index of $H$ in $G$.


\medskip
We now define the \emph{support} maps associated to functors $F\colon\cati\to\cata$ and $F\colon\cati\to\catac$. Consider first the case of a functor $F\colon\cati\to\cata$. Given an arbitrary set $X$, define
$$\supp_X\colon F(X) \longrightarrow \Pf(X)$$
as follows: for $a\in F(X)$, the set $\supp_X(a)$ is the intersection of all $Y\in\Pf(X)$ such that $a$ is in the image of the morphism $F(Y)\to F(X)$. 

One checks that for any injective map $i\colon X\to Z$ one has
$$\supp_Z( F(i)(a))\ \subseteq \ i \,\supp_X(a).$$
Therefore, given any group action $G\acts X$, the map $\supp_X$ is equivariant with respect to the induced action on $F(X)$.
\begin{remarks}
The above inclusion can be strict. This happens for instance for the functor $F_0$ of Example~\ref{ex:trivial} if $A$ is non-trivial, $|X|=1$ and $|Z|>1$. Indeed any non-trivial $a\in F(X)\cong A$ satisfies $\supp_X(a)=X$ and $\supp_Z(F(i)(a)) = \varnothing$.

We also observe that a non-trivial element can have empty support even if $F$ is tight. For instance, let $F$ be a tight functor and take the direct product $F'$ of $F$ with a constant functor associated to a non-trivial group $A$. Then $F'$ is still tight but any element coming from $A$ has empty support.
\end{remarks}
In the case of a functor to $\catac$ the support map is defined in the same way, by replacing $F$ by $F_\set$, thereby yielding a $G$-equivariant map
$$\supp_X\colon F_\set(X) \longrightarrow \Pf(X).$$
We can now state and prove the results extending Theorem~\ref{prop:functor:intro}.
\begin{thm}\label{prop:functor:modern}
Let $F\colon\cati\to\catac$ be any functor and let $G$ be a group acting on a set $X$.

If the action of $G$ on $X$ is extensively amenable, then the action of $F_\grp(X)\rtimes G$ on $F_\set(X)$ is amenable. Moreover it is extensively amenable if $F_\grp$ takes values in $\cata$.

Conversely assume that the action of $F_\grp(X)\rtimes G$ on $F_\set(X)$ is amenable and that $F$ is tight on $X$. Then the action of $G$ on $X$ is extensively amenable.
\end{thm}


\begin{proof}
For each integer $n\geq 0$ we write $[n]= \{j\in \N: 1\leq j \leq n\}$ and choose a mean $m_n$ on $F_\set([n])$ invariant under the (amenable) action of $F_\grp([n]) \rtimes \Sym([n])$. Given any finite set $A$, we obtain a mean $m_A$ on $F_\set(A)$ by transporting $m_n$ for $n=|A|$ through some bijection $[n]\to A$. This mean $m_A$ does not depend on the chosen bijection since $m_n$ is $\Sym([n])$-invariant, and any bijection between two finite sets $A,A'$ maps $m_A$ to $m_{A'}$. We denote by $m_A^X$ the mean on $F_\set(X)$ obtained by push-forward through $F_\set(A)\to F_\set(X)$; this mean is $F_\grp(A)$-invariant, and we have $gm_A^X=m_{gA}^X$ for every $g\in G$.

If $m$ is a $G$-invariant mean on $\Pf(X)$ giving full weight to the collection of subsets that contain any given finite subset, we get a mean $\widetilde m$ as in the proof of Proposition~\ref{prop:functor:cata} by setting for every $f\in \ell^\infty(F_\set(X))$
$$\widetilde m(f) = \int_{\Pf(X)} m_A^X(f) dm(A).$$
This mean is $G$-invariant and $F_\grp(X)$-invariant for the same reason as in the proof of Proposition~\ref{prop:functor:cata}. This proves that the action of $F_\grp(X)\rtimes G$ on $F_\set(X)$ is amenable. Before proving that it is extensively amenable if $F_\grp$ takes values in $\cata$, let us prove the converse part of the statement.

Assume that $F_\grp(X)\rtimes G\acts F_\set(X)$ is amenable and that $F$ is tight on $X$. Given a $G$-invariant mean $m$ on $F_\set(X)$, the $G$-map $\supp_X$  provides us by push-forward with a $G$-invariant mean $\overline m$ on $\Pf(X)$. Let $x_0\in X$. By definition, the value of $\overline m$ on the collection of finite sets containing $x_0$ is $m(B)$, where $B = \big\{b\in F_\set(X) : x_0 \in \supp_X(b) \big\}$. This can be re-written as
$$B= \Big\{b\in F_\set(X) : \forall\,A\in\Pf(X) \text{ with } x_0\notin A: b\notin \mathrm{Im}(F_\set(A)\to F_\set(X))\Big\}.$$
In plain words, $B$ is the complement in $F_\set(X)$ of the image of the direct limit of $F_\set(A)$ over all finite $A$ not containing $x_0$. It is thus the complement of the image of $F_\set(X\setminus\{x_0\})$. Since $F$ is tight, we deduce $m(B)>0$ if $m$ is $F_\grp(X)$-invariant. By Lemma~\ref{lemma=JSbis} we conclude that $G \acts X$ is extensively amenable.

It remains to be proven that the action $F_\grp(X)\rtimes G\acts F_\set(X)$ is extensively amenable whenever $F_\grp$ takes values in $\cata$. By the above it is enough to find a tight functor $F_1 \colon \cati \to \cata$ such that the action of $F_1( F_\set(X) ) \rtimes (F_\grp(X) \rtimes G)$ on $F_1(F_\set(X))$ is amenable. We consider the functor $F_1(X) = (\Z/2\Z)^{(X)}$ and we observe that we can see $F_1$ as a functor $\catset \to \cata$ (for a map $f\colon X \to Y$ between sets, the associated group homomorphism $(\Z/2\Z)^{(X)} \to (\Z/2\Z)^{(Y)}$ maps $\delta_x$ to $\delta_{f(x)}$). By the assumption that $F_{\grp}(Y)$ is amenable for all $Y$ we can consider the  functor $F_1 \diamond F$ as defined in Example~\ref{ex:composition2}. Since $F_1$ commutes with direct limits, we see that $F_1 \diamond F (X)$ is the action of $F_1( F_\set(X) ) \rtimes F_\grp(X)$ on $F_1(F_\set(X))$ not only for finite sets $X$, but also infinite sets. The first part of the statement applied to the functor $F_1 \diamond F$ implies that the action of $F_1( F_\set(X) ) \rtimes (F_\grp(X) \rtimes G)$ on $F_1(F_\set(X))$ is amenable and concludes the proof.
\end{proof}

\section{Probabilistic reformulation and recurrent actions} \label{S: inverted orbit}
\subsection{The inverted orbit}
We now give a more probabilistic reformulation of extensive amenability. 
To simplify the statements we make the assumption that $G=\langle S\rangle$ is finitely generated and acts transitively on $X$ (this is very inessential by Lemma~\ref{lemma=JSbis}). We fix a symmetric probability measure $\mu$ on $G$ with generating support and a base point $x_0 \in X$. We can consider the (left) random walk $(g_n)_{n \geq 0}$ on $G$ defined by $g_0=e$ and $g_{n} = h_n g_{n-1}$ for $n \geq 1$, where $(h_i)_{i \geq 1}$ are independent with law $\mu$.

The inverted orbit $\mathbf O_n$ is then the (random) subset of $X$
\begin{equation} \label{eq=inverted_orbit} \mathbf O_n = \{x_0, g_1^{-1}x_0\cdots, g_n^{-1}x_0\}.\end{equation}
If $G$ is not abelian,  the inverted orbit need not have the same distribution as the directed orbit of the random walk on $X$. The inverted orbit is a well-known object, central to the study of growth and random walks on permutational wreath products, see~\cite{Bartholdi-Erschler1,Bartholdi-Erschler:poissongrowth,Amir-Virag:speedexponent}.

It is sometimes convenient to consider the following variation of the inverted orbit
\begin{equation}\label{eq=inverted orbit'}\mathbf O_n'=\{ x_0, \:h_n x_0,\: h_nh_{n-1} x_0,\ldots,\:h_n\cdots h_1 x_0\}. \end{equation}
For a fixed $n$, $\mathbf O_n$ and $\mathbf O'_n$ have the same distribution, although the joint distributions of the processes $(\mathbf O_n)_n$ and $(\mathbf O'_n)_n$ differ.

The next proposition shows that proving extensive amenability of an action $G\acts X$ boils down to a fine understanding the asymptotic behavior of the distribution of $|\mathbf O_n|$. The third reformulation was suggested to us by Gidi Amir, to whom we are grateful for letting us include it here.

\begin{prop}\label{prop:inverted_orbit} Fix $G \acts X$, $x_0 \in X$ and $\mu$ as before. The following properties are all equivalent to extensive amenability. 
\begin{enumerate}[(i)]
\item\label{item:return_probability} $\lim_{n \to \infty}-\frac1n\log \mathbb E(2^{-|\mathbf O_n|}) = 0$.
\item\label{item:return_probabilitybis} for every $\varepsilon>0$ we have $\mathbb P(|\mathbf O_n|< \varepsilon n) > e^{-\varepsilon n}$ for infinitely many $n$'s.
\item  There exists  a sequence of events $A_n\in\sigma(g_1,\ldots, g_n)$ verifying $-\frac 1n \log \mathbb{P}(A_n)\to 0$, conditioned to which  $\frac 1n \mathbb E( |\mathbf O_n|\: :\: A_n)\to 0.$ 
\end{enumerate}
In particular these conditions do not depend on $\mu$ and $x_0$.
\end{prop}
In the last part $\mathbb E( |\mathbf O_n|\: :\: A_n)= \mathbb E( |\mathbf O_n|1_{A_n})/\mathbb{P}(A_n)$  denotes the expectation of $|\mathbf O_n|$ conditioned to the event $A_n$. 
\begin{proof} In the proof we will freely use~\cite{JM} that $G \acts X$ is extensively amenable if and only if $(\Z/2\Z)^{(X)} \rtimes G \acts (\Z/2\Z)^{(X)}$ is amenable (as follows from Theorem~\ref{prop:functor:modern}). Since we assume that $G$ is finitely generated and acts transitively on $X$, the wreath product $(\Z/2\Z)^{(X)}\rtimes G$ is finitely generated and acts transitively on $(\Z/2\Z)^{(X)}$. Thereby if we fix any non-degenerate symmetric, finitely supported probability measure $\nu$ on $(\Z/2\Z)^{(X)}\rtimes G$ we can consider the Schreier graph $\Gamma$ associated to the action $(\Z/2\Z)^{(X)}\rtimes G\acts(\Z/2\Z)^{(X)}$ with generating set $\supp ( \nu)$ and base-point the trivial configuration $f_0=0_{(\Z/2\Z)^{(X)}}$. The left random walk on $(\Z/2\Z)^{(X)}\rtimes G$ with step measure $\nu$ then induces a nearest neighbour random walk $(f_n)$ on $\Gamma$. By Kesten's amenability criterion for a graph (see~\cite[Theorem 10.6]{Woess}), amenability of the action $(\Z/2\Z)^{(X)}\rtimes G\acts (\Z/2\Z)^{(X)}$ is thereby equivalent to 
\begin{equation}\label{E: return prob gamma}
\lim_{n\to \infty}-\frac 1n \log \mathbb P(f_n=f_0)=0.\end{equation}
Moreover it is sufficient to show this for any choice of $\nu$ with the above properties. We can thereby choose $\nu$ to be the \emph{switch-walk-switch} measure $\widetilde \mu$ on $(\Z/2\Z)^{(X)}\rtimes G$ associated to the measure $\mu$ on $G$ and to the base-point $x_0\in X$. Recall that if $\lambda$ is the uniform probability measure on $\{0_{(\Z/2\Z)^{(X)}},\delta_{x_0}\}\subset (\Z/2\Z)^{(X)}$ then by definition $\widetilde \mu= \lambda \ast \mu \ast \lambda$, where  $\lambda$ and $\mu$ are naturally seen as probability measures on $(\Z/2\Z)^{(X)}\rtimes G$. With this choice of $\nu$ it is then easy to see that $f_{n+1}$ is obtained from $f_n$ through the following steps: first change the value of $f_n(x_0)$ to a uniform random value in $\Z/2\Z$ to obtain a new $f'_n$, then translate $f'_n$ to $f''_n=h_{n+1}\cdot f'_n$ where $h_{n+1}\in G$ has distribution $\mu$, finally randomize again $f''_n(x_0)$ to obtain $f_{n+1}$. It follows from this description that $\supp f_n\subset \mathbf O'_n$ for every $n$, where $\mathbf O'_n$ is as in \eqref{eq=inverted orbit'}. Moreover $f_n$ lights each point in $\mathbf O'_n$ independently with probability 1/2.  From this we immediately get that $\mathbb P(f_n=f_0) = \mathbb E(2^{-|\mathbf O'_n|})=\mathbb E(2^{-|\mathbf O_n|})$. The equivalence between \eqref{item:return_probability} and extensive amenability follows.

The equivalence between \eqref{item:return_probability} and \eqref{item:return_probabilitybis} is essentially Markov's inequality.

The second condition implies the third by a diagonal extraction argument.

To see that the first condition implies the first, observe that $2^{-|\mathbf O_n|}\geq 2^{-|\mathbf O_n|}1_{A_n}$ and thereby by Jensen's inequality
\[-\frac 1n \log \mathbb E(2^{-|\mathbf O_n|}
)\leq -\frac 1n\log\mathbb E (2^{-|\mathbf O_n|}1_{A_n})\leq  -\frac 1n \log \mathbb{P}(A_n)-\frac1n \mathbb E( |\mathbf O_n|\: :\: A_n)\log 2\to 0.\qedhere\]
\end{proof}
\subsection{Application to recurrent actions} \label{S: recurrent}

As an application let us give a new proof of a criterion for extensive amenability in~\cite{JNS}. Here a group action $G\acts X$ is called \emph{recurrent} if for every symmetric, finitely supported probability measure $\mu$ on $G$ and every $x_0\in X$, the random walk $(g_nx_0)$ on the orbit of $x_0$ induced by the left random walk on $G$ is recurrent. If $G$ is finitely generated, it is sufficient to check this for a symmetric, finitely supported probability measure with generating support. Equivalently it is sufficient to check that for every $x\in X$ and for a symmetric, finite generating set $S$ of $G$ the Schreier graph $\Gamma(G, x_0, S)$ is recurrent for simple random walk.

\begin{thm}[Theorem 1.2 in~\cite{JNS}]\label{T: recurrent}
Recurrent actions are extensively amenable.
\end{thm}
Let us give a direct proof based on the inverted orbit. The main ingredient is
\begin{lemma}[\cite{Bartholdi-Erschler:poissongrowth},\cite{Amir-Virag:speedexponent}]\label{L: recurrence inverted orbit}
Assume that $G$ is finitely generated and the action $G\acts X$ is transitive. Then $G\acts X$ is recurrent if and only if $\frac 1n \mathbb E|\mathbf O_n|\to 0$ for some (equivalently for any) non-degenerate symmetric, finitely supported probability measure $\mu$ on $G$.

\end{lemma}
Although this is exactly~\cite[Lemma 3.1]{Bartholdi-Erschler:poissongrowth} we report a proof for the convenience of the reader.

\begin{proof}

Let $T=\min\{n\geq 1\: :\: {g_n}x_0=x_0\}\in\N\cup\{\infty\}$. We have
\begin{align*}\mathbb E |\mathbf O_{n+1}|-\mathbb E|\mathbf O_n|= \mathbb P(g_{n+1}^{-1}x_0\notin \mathbf O_n)=\\
\mathbb P(h_{n+1}^{-1}x_0\neq x_0,\: h_n^{-1}h_{n+1}^{-1}x_0\neq x_0,\ldots,\: h_1^{-1}\cdots h_{n+1}^{-1}x_0\neq x_0)=\mathbb P(T>n+1),\end{align*}
since by symmetry $(h_{n+1}^{-1},\:h_n^{-1}h_{n+1}^{-1},\ldots,\: h_1^{-1}\cdots h_{n+1}^{-1})$ has the same law as $(g_1, \ldots, g_{n+1})$. This computation shows that $\frac1n \mathbb E|\mathbf O_n|\to \mathbb P(T=\infty)$, which vanishes if and only if $G\acts X$ in recurrent.\qedhere
\end{proof}
\begin{proof}[Proof of Theorem~\ref{T: recurrent}]
By Lemma~\ref{lemma=JSbis} we may assume that $G$ is finitely generated and acts transitively on $X$. Apply part 1 of Proposition~\ref{prop:inverted_orbit}. By convexity 
\[-\frac 1n\log \mathbb E(2^{-|\mathbf O_n |})\leq\frac 1n \mathbb E ({|\mathbf O_n|}) \log 2\to 0,\]
where we used Lemma~\ref{L: recurrence inverted orbit}. Part~\eqref{item:return_probability} of Proposition~\ref{prop:inverted_orbit} gives the conclusion. 
\end{proof}

\section{Interval exchange transformations}\label{section:iet}

\subsection{Amenability of subgroups of low rational rank}
Let $\Lambda<\R/\Z$ be a finitely generated subgroup of the circle. Thus, $\Lambda$ is isomorphic to $\Z^{d} \times F$ for some finite abelian group $F$ and some integer $d$, the \emph{rational rank} of $\Lambda$, denoted $\rk(\Lambda)$. With this notation the rational rank (as defined in Definition \ref{def=rational_rank_iet}) of a finitely generated subgroup $G \le \iet$ is an abbreviation for $\rk(\Lambda(G))$. We denote $\iet(\Lambda)$ the subgroup of all $g\in \iet$ so that the angle $gx-x$ is in $\Lambda$ for every $x\in \R/\Z$.

The following result is a reformulation of Theorem~\ref{T:rank2}.
 
 \begin{thm}\label{T:rank2full}
 Let $\Lambda<\R/\Z$ be finitely generated. If $\rk(\Lambda)\leq 2$, then $\iet(\Lambda)$ is amenable.
\end{thm}

Our goal is to apply Corollary~\ref{prop:functor:amenable:intro} for the functor $F\colon \Sym\colon \cati \to \cata$, given by $F(A)=\Sym(A)$ the symmetric group. This extends to an infinite set as $F(X)=\Sym(X)$, the group of permutations of $X$ with finite support (see Example~\ref{ex:sym}). If $G \acts X$ we recall that we have an action of $G$ on $\Sym(X)$ by conjugation. For $\tau \in \Sym(X)$ and $g \in G$ we denote ${}^g\tau = g \tau g^{-1}$. The first ingredient is (see 
\S~\ref{S: recurrent} for the definition of a recurrent action)
\begin{lemma} If $\rk(\Lambda)\leq 2$, the action of $\iet(\Lambda)$ on $\R/\Z$ is recurrent. In particular, it is extensively amenable.
\end{lemma}

\begin{proof}

Equip $\Lambda$ with a finite symmetric generating set and the corresponding Cayley graph structure. Then $\Lambda$ is a recurrent graph. Let $H<\iet(\Lambda)$ be a finitely generated subgroup, equipped with a finite symmetric generating set $S$. For $x\in\R/\Z$ let $\Gamma(x, H, S)$ be the orbital Schreier graph for the action of $H$ on $H_x$. We have an injective map
 \begin{align*}   \Gamma(x, H, S)&\to \Lambda\\
  y&\mapsto y-x
\end{align*}
where the difference is taken in $\R/\Z$. This maps takes values in $\Lambda$ since $y=hx$ for some $h \in H<\iet(\Lambda)$. It is not hard to check that this map is Lipschitz (for some constant depending on the generating set $S$ of $H$ only). Since $\rk(\Lambda)=2$ this implies that the action is recurrent, see ~\cite[Theorem 2.17]{LyonsPeres}. The action is then extensively amenable by Theorem~\ref{T: recurrent}.
\end{proof}

The other ingredient in the proof of Theorem~\ref{T:rank2full} is the following result that concludes the proof.
\begin{prop}
A subgroup $G\le \iet$ is amenable if and only if the action $G\acts \R/\Z$ is extensively amenable.
\end{prop}
\begin{proof} The ``only if part'' is obvious (Lemma~\ref{lemma=intermediatepty}). 

The "if part" is based on the method explained in Remark \ref{R: method}.
  If we replace the convention that interval exchanges are right-continuous by the condition that they are left continuous, we get a group of permutations of $\R/\Z$ that we denote by $\widetilde{\iet}$. The map $g \in \iet \mapsto \widetilde g \in\widetilde{\iet}$, where $\widetilde g$ is the unique left-continuous map that coincides with $g$ except on the points of discontinuity of $g$, is a group isomorphism. Then 
\begin{equation}\label{def:taug}\tau_g = \widetilde g g^{-1}\end{equation} is a permutation of $\R/\Z$ with finite support equal to the points of discontinuity of $g^{-1}$. Moreover 
\begin{equation}\label{E: cocycle}\tau_{gh} = \widetilde g \widetilde h h^{-1} g^{-1} = \tau_g g \tau_h g^{-1} = \tau_g({}^g\tau_h),\end{equation}
 so that the map 
  \begin{align*}
  \iota\colon  \iet&\to \Sym(\R/\Z)\rtimes \iet\\
  g&\mapsto (\tau_{g}, g)
  \end{align*}
  is an injective group homomorphism.
Observe that $\tau_g=1$ if and only if $g$ is continuous, i.e. a rotation. If $G\le\iet$ is a subgroup, the restriction of $\iota$ to $G$ takes values in $\sym(\R/\Z)\rtimes G$. Moreover $\iota(G)\cap\{1\}\times G$ consists of rotations and thereby is amenable. In the terminology of Remark \ref{R: method}, $g\mapsto \tau_g$ is a $\sym(\R/\Z)$-cocycle with amenable kernel.
 The conclusion follows from Corollary~\ref{prop:functor:amenable:intro}.\end{proof}

\subsection{Non-trivial boundary for subgroups of high rank}\label{S: Liouville}

 We take notations from the previous section. Let again $\Lambda$ be a finitely generated subgroup of $\R/\Z$. Note that $\Lambda$ identifies with an abelian subgroup of $\iet(\Lambda)$ acting on $\R/\Z$ by rotations.
 
 \begin{thm}\label{T: poisson}
Let $\Lambda< \R/\Z$ be a finitely generated group.

\item[1.] If $\rk(\Lambda)=1$, then $\iet(\Lambda)$ has the Liouville property for every symmetric, finitely supported measure.
\item[2.] Assume $\rk(\Lambda)\geq 3$. Let $G\le\iet(\Lambda)$ be a non-abelian group that contains $\Lambda$. Then any finitely supported, non-degenerate probability measure $\mu$ on $G$ has non-trivial Poisson--Furstenberg boundary.
 \end{thm}
We do not know whether part 1 can be extended to the case $\rk(\Lambda)=2$. 
 
The proof of part 1 is postponed to \S~\ref{S: top full group}. To prove part 2, we will use the following fact that allows to restrict to symmetric measures. Given a probability measure $\mu$ on a countable group $G$, we denote $\check{\mu}$ the  measure $\check{\mu}(g)=\mu(g^{-1})$

\begin{lemma}[Baldi, Lohou\'e and Pery\`ere~\cite{Baldi-Lohoue-Peryiere}] \label{L: non symmetric transient}
Let $G$ be a countable group acting transitively on a set $X$ and $\mu$ be a probability measure on $G$. Consider the symmetric measure $\nu=\frac12 (\mu+\check{\mu})$. Assume that the Markov chain on $X$ induced by $\nu$ is transient. Then so is the Markov chain induced by $\mu$.
\end{lemma}
\begin{proof} For $G=X$ acting on itself, this is exactly~\cite[Proposition 1]{Baldi-Lohoue-Peryiere}, and the proof extends with no changes to group actions.
\end{proof}
 
 Let $(G, \mu)$ be as in part 2 of Theorem~\ref{T: poisson}.
 For the proof, it will be convenient to consider the \emph{right} random walk $g_n=h_1\cdots h_n$ with step measure $\mu$. When speaking about recurrence or transience of left group actions we refer however to left random walk. Thereby with the right random walk notations, saying that the action of $G$  on the orbit of $x_0\in\R/\Z$ is transient means that the Markov chain $(g_n^{-1}x_0)$ is transient.
 Observe that this is the case under the assumptions of part 2 of Theorem~\ref{T: poisson}.

 For $g\in \iet$ let $\tau_g\in \sym(\R/\Z)$ be the permutation introduced in the proof of Proposition~\ref{T:rank2full}. 

\begin{lemma}
 For every $x\in \R/\Z$ the value $\tau_{g_n}(x)\in \R/\Z$ is almost surely constant for large enough $n$. We denote $\tau_\infty(x)$ its eventual value. The map $\tau_\infty\colon \R/\Z\to \R/\Z$ is an injective map that preserves every $\Lambda$-coset.
 \end{lemma}
 
  \begin{proof}
  Let $\Sigma\subset \R/\Z$ be the union of $\supp \tau_s$ for $s\in \supp\mu$. By transience, for every $x\in \R/\Z$ there exists a random $N$ so that $x\notin g_n\Sigma$ for $n\geq N$. Now observe that by the cocycle relation \eqref{E: cocycle} we have
   \[\tau_{g_n}=\tau_{h_1}\circ{}^{g_1} \tau_{h_2}\cdots\circ{}^{g_{n-1}} \tau_{h_n}\]
 and that the support of ${}^{g_{i-1}}\tau_{h_i} $ is contained in $g_{i-1}\Sigma$. Hence $\tau_{g_n}(x)=\tau_{g_N}(x)$ for every $n\geq N$. The fact that $
 \tau_\infty$ is injective follows immediately from the fact that $\tau_{g_n}$ is injective for every $n$, and it preserve every $\Lambda$-coset in $\R/\Z$ since so does $G$.
  \end{proof}
  Denote $\operatorname{Inj}_\Lambda(\R/\Z,\R/\Z)$ the set of all injective maps of the circle to itself that preserve every $\Lambda$ coset. The group $\iet(\Lambda)$ acts on $\operatorname{Inj}_\Lambda(\R/\Z,\R/\Z)$ by 
   \[g\cdot f=\tau_g\circ{}^g f=\tau_g\circ g\circ f\circ g^{-1},\]
where $g\in \iet(\Lambda)$ and $f\in \operatorname{Inj}_\Lambda(\R/\Z,\R/\Z)$.
Endow with the topology induced by the product of the \emph{discrete} topology on $\R/\Z$. 
 Then the above construction shows that $\tau_{g_n}$ converges almost surely to a $\tau_\infty\in\operatorname{Inj}_\Lambda(\R/\Z,\R/\Z)$. If $\nu$ is the distribution of $\tau_\infty$ then $(\operatorname{Inj}(\R/\Z,\R/\Z), \nu)$ is a quotient of the Poisson--Furstenberg boundary. To conclude the proof we only need to check that this quotient is not the trivial one, namely that $\nu$ is not concentrated on a single point.
 
 \begin{lemma}\label{L: abelian stabilizers}
Let $f\in \operatorname{Inj}_\Lambda(\R/\Z,\R/\Z)$ and $G$ be as in part 2 of Theorem~\ref{T: poisson}. Then there exists $g\in G$ so that $g\cdot f\neq f$.
\end{lemma}
\begin{proof}
Assume that $G$ stabilizes $f$. A rotation $r\in \Lambda<G$ acts on $f$ by conjugation. Hence the increment $x\mapsto f(x)-x$ is $\Lambda$-invariant. It follows that $f$ coincides on every $\Lambda$-coset with a translation by a fixed element in $\Lambda$. Observe that this implies the following: if $g\in G$ is arbitrary, then $g \circ f \circ g^{-1}$ is either equal to $f$ or it differs from $f$ on infinitely many points. Indeed, observe that the restriction of $g\circ f \circ g^{-1}$ on every $\Lambda$-coset coincides with the restriction of an interval exchange transformation, and two interval exchange transformations are either equal or differ on infinitely many points in every $\Lambda$-coset (since these are dense).
Now consider $g\in G$ which is not a rotation. For such a $g$ we have $\tau_g\neq \operatorname{Id}$. Hence $g\cdot f=\tau_g\circ g\circ f \circ g^{-1}$ cannot be equal to $f$, in the first case since $\tau_g$ is nontrivial, and in the second case since it is finitely supported. \qedhere
\end{proof}

\begin{proof}[Proof of Theorem~\ref{T: poisson} (part 2)]  Assume that $\tau_\infty$ is deterministic. Then the Dirac mass on $\tau_\infty$ is a $\mu$-stationary measure. This implies that the support of $\mu$ stabilizes $\tau_\infty$, hence so does $G$ since $\mu$ is non-degenerate. But this is impossible by Lemma~\ref{L: abelian stabilizers} \qedhere

 \end{proof}
 
 \begin{question}
With the above notations, is  $(\operatorname{Inj}_\Lambda(\R/\Z,\R/\Z), \nu)$ isomorphic to the Poisson--Furstenberg boundary of $(G, \mu)$?
 
 \end{question}

  Similar ideas yield the following alternative for subgroup containing an irrational rotation, that could also be proven more directly.
  \begin{prop}
  Let $G$ be a finitely generated subgroup of $\iet$ that contains an irrational rotation. Then either $G$ is an abelian group of rotations, or $G$ has exponential growth.
  
  \end{prop}
  
  \begin{proof}
  Suppose that $G$ is nonabelian. Let $r\in G$ be an irrational rotation and $h\in G$ be an element which is not a rotation, so that $\tau_h\neq 0$. Consider the non-symmetric measure $\mu=\frac{1}{3}\delta_r+\frac{1}{3}\delta_{h^{-1}}+\frac{1}{3}\delta_h$. The random walk on the Schreier graphs of the set of angles is drifted in the direction of $r$ (perhaps after replacing $r$ with a big enough power). The same idea as in the previous proof shows that $\mu$ has non-trivial Poisson--Furstenberg boundary. Existence of a finitely supported measure with non-trivial boundary implies that $G$ has exponential growth, see~\cite[Proposition 1.4]{Kaimanovich-Vershik}.
  
  \end{proof}

\subsection{Connection with topological full groups and Liouville property in rank 1} \label{S: top full group}
In this section we show part 1 of Theorem~\ref{T: poisson}.  We first explain in some details a construction,  that has already appeared implicitly in~\cite{Cornulier:Bourbaki}, showing that the study of finitely generated subgroups of $\iet$ is closely related to the study of the \emph{topological full groups} of a special family of minimal actions of finitely generated abelian groups on the Cantor set.

Let $\Gamma$ be a group acting by
homeomorphisms on a topological space $\X$. The \emph{topological full group} of the action, $[[\Gamma]]$, is 
the group of all homeomorphisms $h$ of $\X$ such that
every point of $\X$ admits a neighborhood where $h$ agrees with an
element of $\Gamma$. 

The dynamical system $(\Gamma, \X)$ is \emph{minimal} if there are no non-trivial closed $\Gamma$-invariant subsets in $\X$.

A \emph{Cantor  $\Gamma$-system} is a dynamical system $(\Gamma, \X)$ where $\X$ is the Cantor set.
Let $A$ be a finite alphabet. The \emph{$\Gamma$-shift} over $A$ is the Cantor system $(\Gamma, A^\Gamma)$ where $\Gamma$ acts on $A^\Gamma$ by translations. A \emph{$\Gamma$-subshift} is a Cantor $\Gamma$-system $(\Gamma, \X)$ where $\X\subset A^\Gamma$ is a closed $\Gamma$-invariant subset.

In what follows we use the notation $\vee_{\alpha\in I} \mathcal{P}_\alpha$ for the join  of a family of partitions $(\mathcal P_\alpha)_{\alpha\in I}$ of a set. 
Recall the following elementary criterion to establish whether a Cantor $\Gamma$-system is conjugate to a subshift. For a proof see e.g.~\cite[Fait 2.2]{Cornulier:Bourbaki}. 
\begin{lemma}\label{L: subshift}
Let $(\Gamma, \X)$ be a Cantor $\Gamma$-system. then $(\Gamma, \X)$ is conjugate to a $\Gamma$-subshift over a finite alphabet if and only if there exists a finite partition $\mathcal{P}$ of $\X$ into clopen sets so that the partition $\vee_{\gamma\in \Gamma} \gamma \mathcal P$ is the point partition of $\X$ (i.e. for any $x\neq y\in \X$ there exists $\gamma\in \Gamma$ so that the partition $\gamma\mathcal{P}$ separates $x$ and $y$).
\end{lemma}

As before, let $\Lambda<\R/\Z$ be a finitely generated and infinite subgroup, and let $\Sigma=\{x_1,\ldots, x_r\}$ be a finite subset of $\R/\Z$. We denote $\iet(\Lambda;\Sigma)$ the subgroup of $\iet(\Lambda)$ consisting of interval exchange transformations so that all extrema of the defining intervals lie in cosets $x_i+\Lambda$ of points in $\Sigma$. Every finitely generated $G<\iet$ is contained in $\iet(\Lambda, \Sigma)$, where $\Lambda=\Lambda(G)$ and $\Sigma$ is the set of extrema of the defining intervals of elements in some generating set of $G$.

A construction going back essentially to~\cite{keane} realizes $\iet(\Lambda,\Sigma)$ as a group of homeomorphisms of a Cantor set. Using Lemma \ref{L: subshift}, we will see through this construction that $\iet(\Lambda; \Sigma)$ is isomorphic to the topological full group of a minimal $\Lambda$-subshift, and even that the actions are semiconjugate.

\begin{prop}\label{prop:cantor}
Let $\Lambda<\R/\Z$ be infinite and finitely generated and $\Sigma$ as above. There is a minimal $\Lambda$-subshift $(\Lambda,X)$, an isomorphism $\pi \colon [[\Lambda]] \to \iet(\Lambda; \Sigma)$ and a continuous surjective map $h \colon X \to \R/\Z$ such that $h(g \cdot x) = \pi(g) \cdot x$ for all $x \in X$ and $g \in [[\Lambda]]$.
\end{prop}
\begin{remark}
The Cantor minimal $\Lambda$-systems arising in this way have a very special form. Recall that the topological full group of a Cantor minimal $\Z^d$-system may fail to be amenable even for $d=2$~\cite{Elek-Monod}, while this does not happen for $\iet$ as shown by Theorem \ref{T:rank2full}.
\end{remark}

\begin{proof}
Let $\mathcal{C}$ be the space constructed out of the circle in the following way. Take a circle and double all points in $\Sigma+\Lambda$. Namely, we replace each point  $x\in \Sigma+\Lambda$ by two points: left and right, that we denote $x_-$ and $x_+$. There is a natural topology on $\mathcal{C}$ that makes it homeomorphic to the Cantor set. The action of $\Lambda$ on $\R/\Z$ by rotations induces an action on $\mathcal C$ by homeomorphisms. This action is minimal as soon as $\Lambda$ is infinite. It is then easy to see that the map $x_\pm \in X \mapsto x \in \R/\Z$ is a continuous surjection that implements a semiconjugation between the action of $[[\Lambda]]$ and $\iet(\Lambda; \Sigma)$.

To see that the $\Lambda$-Cantor system is in fact a subshift, we apply Lemma~\ref{L: subshift} to the following partition. Let $S$ be a generating set of $\Lambda$. For each $\lambda\in S$ and $x\in \Sigma$ consider the partition $\mathcal P_{\lambda, x}$ of $\mathcal C$ in the two clopen sets $[x_+,(x+\lambda)_-]$ and $[(x+\lambda)_+, x_-]$. Set $\mathcal P=\vee_{\lambda, x} \mathcal P_{\lambda, x}$. Then the $\Lambda$-translates of $\mathcal P$ separate points. 
\end{proof}

We now explain how this implies Part 1 of Theorem~\ref{T: poisson}. We first recall the notion of \emph{complexity} of a subshift.  Let $\Gamma$ be a finitely generated group with generating set $S$ and $\X$ be a $\Gamma$-subshift. Let $\mathcal P$ be a partition of $\X$ satisfying Lemma~\ref{L: subshift}. The \emph{one-dimensional complexity} of the subshift $(\Gamma, \X)$ with respect to $S$ and $\mathcal P$ is the function $\rho_{S, \mathcal P}\colon \N\to \N$ that counts the number of elements of the partition $\vee_{|\gamma|_s\leq n}\gamma \mathcal P$, where $|\cdot |_S$ is the word metric associated to $S$. For a $\Z$-subshift $\X\subset A^\Z$ this coincides with the more standard definition of complexity of sequences if one takes the standard generating set of $\Z$ and the partition indexed by $A$ into cylinder sets corresponding to the letter at position 0. The proof of the following lemma is elementary and not difficult.

\begin{lemma}\label{L: complexity}
 Let $(\Lambda, \mathcal C$) the Cantor minimal system constructed in the previous proposition, and assume that $\rk(\Lambda)=d$. Then for every generating set $S$ of $\Lambda$ and every partition $\mathcal P$ satisfying Lemma~\ref{L: subshift} there exists a constant $C>0$ so that $\rho_{S, \mathcal P}(n)\leq Cn^d$.
\end{lemma}

\begin{proof}[Proof of Part 1 of Theorem~\ref{T: poisson}]
The conclusion follows from~\cite[Theorem 1.2]{Matte} by Proposition~\ref{prop:cantor} and Lemma~\ref{L: complexity}. (The statement of~\cite[Theorem 1.2]{Matte} assumes that $\Lambda=\Z$, but the proof extends with no changes if $\Lambda$ is virtually cyclic).\qedhere

\end{proof}

\section{A hereditarily amenable action which is not extensively amenable}\label{section:frankenstein}
Corollary~\ref{C: hereditarily amenable} shows that an obstruction for an amenable action to be extensively amenable is to fail be \emph{hereditarily amenable}. In this section we show that this is not the only obstruction.

We consider a group $H$ of piecewise projective orientation-preserving 
homeomorphisms of $\R$ intruduced in~\cite{FM}. A self-homeomorphisms $f$ of $\R$ belongs to this group if there exist intervals $I_1,\dots,I_n$ 
covering $\R$ and $g_1,\dots,g_n \in \mathrm{PSL}(2,\R)$ such that $f$ 
coincides with $g_i$ on $I_i$ for all $i \in \{1,\dots,n\}$. Here we 
consider the standard projective action of $\mathrm{PSL}(2,\R)$ on 
$P^1(\R) = \R \cup \infty$: $g \cdot x= \frac{ax+b}{cx+d}$ if $g = 
\pm\begin{pmatrix} a & b \\ c & d\end{pmatrix}$.
The goal of this section is to show

\begin{thm}
The action $H\acts \R$ is hereditarily amenable, but it is not extensively amenable.
\end{thm}

\begin{remark}
By Lemma~\ref{lemma=JSbis}, this implies that there exists a finitely generated subgroup $L\le H$ and an $L$-orbit $X\subset \R$ so that the action of $L$ on $X$ is hereditarily amenable but not extensively amenable.
\end{remark}

The first ingredient of the proof is:

\begin{lemma}\label{lemma:Frankenstein_hereditarily_amenable}
$H \acts \R$ is hereditarily amenable.
\end{lemma}

\begin{proof}
Let $H^{(0)} \le H$ be any subgroup and let $x \in \R$. Denote by $H^{(1)}$ the 
commutator subgroup of $H^{(0)}$, and by $H^{(2)}$ the commutator subgroup of 
$H^{(1)}$. Denote $M=\sup(H^{(0)} x) \in \R \cup \{\infty\}$. It is a fixed 
point of $H^{(0)}$. Let $x_n \in H^{(0)} x$ be a sequence converging to $M$. Since 
the double commutator of the stabilizer of $M$ in $\mathrm{PSL}(2,\R)$ 
is trivial, every element $h$ of $H^{(2)}$ is trivial on a neighborhood 
of $M$, and hence satisfies $h x_n = x_n$ for all $n$ large enough 
depending on $h$. Therefore any w-* cluster point of the sequence 
$\delta_{x_n}$ is an $H^{(2)}$-invariant mean on $H^{(0)}x$.

This shows that the compact convex subset $K \subset \ell^\infty(H^{(0)}x)^*$ 
of all $H^{(2)}$-invariant means on $X$ is nonempty. The action of $H^{(0)}$ 
on $K$ factors to an action of the amenable group $H^{(0)}/H^{(2)}$, and hence 
has a fixed point. Such a fixed point is an $H^{(0)}$-invariant mean on $X$.
\end{proof}

The fact that $H \acts \R$ is not extensively amenable follows from the 
following Theorem because it was proved in~\cite{FM} that $H$ is not amenable.

\begin{thm}\label{thm:Frankenstein} A subgroup $H_1$ of $H$ is amenable 
if and only if $H_1 \acts \R$ is extensively amenable.
\end{thm}

The proof relies on the main result in~\cite{JNS}. Recall that given a topological space $\X$ and a groupoid of germs of homeomorphims $\G$ acting on $\X$, the \emph{topological full group} $[[\G]]$ of $\G$ is the group of all self-homeomorphisms of $\X$ whose germs belong to the groupoid $\X$. 

\begin{thm}[\cite{JNS}]\label{T: JNS}
Let $G$ be a group acting on a topological space $\X$ with groupoid of germs $\G$. Assume that there is a groupoid of germs of homeomorphisms $\H$ acting on $\X$ so that the following holds 
\begin{enumerate}[(i)]
\item For every $g\in G$ the germ of $g$ at $x$ belongs to $\H$ for all but finitely many $x\in \X$.
\item For every $x\in X$ the isotropy group $\mathcal G_x$ is amenable.
\item The action $G\acts \X$ is extensively amenable.
\item The group $[[\H]]$ is amenable.

\end{enumerate}

Then $G$ is amenable. 

\end{thm}
In fact, condition~(iii) in the original statement in~\cite{JNS} was that the action is recurrent, but this is used in the proof only through Theorem~\ref{T: recurrent}. 

\begin{proof}[Proof of Theorem~\ref{thm:Frankenstein}]
The only if part holds in view of Lemma~\ref{lemma=intermediatepty}. Assume that $H_1\acts \R$ is extensively amenable. Apply Theorem~\ref{T: JNS} to $G=H_1$ with $\mathcal H$ the groupoid of germs of the partial action of $\operatorname{PSL}(2,\R)$ on the real line. Condition~(ii) is satisfied since $\mathcal G_x\simeq \operatorname{Aff}(\R)\times\operatorname{Aff}(\R)$. To check condition~(iv), let $h\in [[\mathcal H]]$. Since projective homeomorphisms are analytic, it is  easy to see that the germ of $h$ in  at any two points $x,y\in \R$ is represented by the same element of $\operatorname{PSL}_2(\R)$. Therefore $h\in \operatorname{PSL}(2, \R)$ and $h$ stabilizes $\R$ globally. This shows that $[[\mathcal H]]=\operatorname{Aff}(\R)$ is amenable. Thereby $H_1$ is amenable. \end{proof}

 \section{Unrestricted wreath product actions}\label{section:unrestricted}
 This section is a complement to \S~\ref{section:functors}. In the results therein, we restricted ourselves to functors that can be written as direct limits of functor defined on finite sets. The next proposition shows that Theorem~\ref{prop:functor:intro} fails if this assumption is removed.
 
  \begin{prop}\label{P: counterexample unrestricted}
 There exists an extensively amenable, transitive action $G\acts X$ of a finitely generated group so that the action of the unrestricted wreath product $(\Z/2\Z)^X\rtimes G$ on $(\Z/2\Z)^X$ is not amenable.
 \end{prop}
 Here $(\Z/2\Z)^X$ denotes the abelian group of all configurations $f\colon X\to \Z/2\Z$.

We will use the following simple fact.
 \begin{lemma}\label{L: non-amenability orbits}
Let $G$ be a finitely generated group acting on a set $Y$, and let $\mu$ be a non-degenerate, symmetric finitely supported probability measure on $G$. Then the action of $G$ on $Y$ is non-amenable if, and only if, the spectral radius of $\mu$ on every $G$-orbit is uniformly bounded sway from 1, i.e. if there exists a uniform $\rho<1$ so that for any $y_0\in Y$ we have $\lim_n\mathbb{P}(g_{2n}y_0=y_0)^{1/2n}<\rho$, where $G_n$ is the left random walk on $G$ with step measure $\mu$.
\end{lemma} 
\begin{proof}
The action is nonamenable if and only if the spectral radius (i.e. the norm of the convolution operator  $f\mapsto \mu*f$ on $\ell^2(Y)$) is smaller than 1. This is equivalent to the fact that the spectral radius  on every $G$-orbit is uniformly bounded away from 1. It is well-known that the spectral radius  on the orbit of $y_0$ equals $\lim_n\mathbb{P}(g_{2n}y_0=y_0)^{1/2n}$, see~\cite{Woess}.
\end{proof}

To show Proposition~\ref{P: counterexample unrestricted} we construct a finitely generated group $G=\langle S\rangle$ acting faithfully and transitively on a set $X$ so that 
\begin{enumerate}
\item the Schreier graph of the action is isomorphic to $\Z$;
\item the action of $(\Z/2\Z)^X\rtimes G$ on $(\Z/2\Z)^X$ is not amenable.
\end{enumerate}

Note that the first condition guarantees that the action is extensively amenable by Theorem~\ref{T: recurrent}.

 Set $G=\Z/2\Z*\Z/2\Z*\Z/2\Z$, freely generated by $3$ involutions $b, y, r$ and set $X=\Z$. Let $B, R, Y$ denote three colors (blue, red and yellow). Consider the space of \emph{proper colorings} of edges of the line $\Z$, i.e. colorings where two adjacent edges never have the same color. Every such coloring defines an action of $G$ on $X$, where the generators $b,y,r$ act as involutions by switching two adjacent integers whenever the edge between them has the corresponding color $B, Y, R$. Consider any proper coloring  so that every  finite proper word in $B,R,Y$ appears infinitely many times (for example endow the space of of proper colorings with the natural Markov measure and pick a random one). This coloring gives the action $G\acts X$ that we consider (this is essentially the same action constructed in~\cite{VanDouwen}). We will abuse notations and identify a proper word in the colors $B,R,Y$ with an element of the group $G$.

Let $f\in (\Z/2\Z)^X$ to be determined later and consider the group $H \le (\Z/2\Z)^X\rtimes G$ generated by $G$ and $f$. We will construct an element $f$ such that the action of $H$ on any $H$-orbit in $(\Z/2\Z)^X$ has spectral radius uniformly bounded away from 1. This concludes by Lemma~\ref{L: non-amenability orbits}.

To construct $f$, arrange all reduced words in B,R,Y in a list (i.e. consider a numeration of elements of $G$, but we avoid introducing the numeration explicitly). Consider the first  word in the list $w$, and let $l=|w|$ be its length. Pick an auxiliary  word $\tilde{w}$ so that $|\tilde{w}|\geq 5 l$.  Find $l+1$ copies of the word $\tilde{w}w\tilde{w}$ in the colored line $X$ that do not overlap. For $i=0,\ldots, l$  let $x_i\in X$ be the vertex between the first $\tilde{w}$ and $w$ in the $i$-th copy of $\tilde{w}w\tilde{w}$.

Consider a big interval $I_w$ that contains all these words, and define  $f$ on $I_w$  to be $f(x_i+i)=1$ for every $i=0,\ldots,l$ and $f(y)=0$ for all other $y\in I_1$. Here $x_i+i$ denotes the point lying $i$ positions to the right of $x_i$

Then consider the second word $w'$ in the list and repeat the same construction taking care that everything happens outside of $I_w$, define a bigger interval $I_{w'}$ containing $I_w$, and so on. This defines $f$ everywhere. 
In fact, a random $f\in (\Z/2\Z)^X$ distributed according the uniform Bernoulli measure will also have the features that we need.

For every word $w\in G$, we denote $x_0(w),x_1(w),\cdots x_{|w|}(w)\in X$  the points $x_i$ appearing in the construction.

Let now $\nu$ be the equidistributed measure on the standard generating set of $G$, set $\eta=\frac{1}{2}\delta_f+\frac{1}{2}\delta e$, and consider the ``switch-walk-switch like'' measure on $H\le(\Z/2\Z)^X\rtimes G$ given by $\mu=\eta*\nu*\eta$.  Let $g_n=(f_n, w_n)$ be the right random walk on $H$ with step measure $\mu$, where $f_n \in (\Z/2\Z)^X$ and $w_n\in G$. Let also $h_1,\cdots h_n$ be the increments of $(w_j)$, so that $w_n=h_n\cdots h_1$.

Let $t\in (\Z/2\Z)^X$ be arbitrary, and set $t_n=g_n\cdot t$. The following lemma shows that the action of $H$ on $(\Z/2\Z)^X$ is nonamenable by Lemma~\ref{L: non-amenability orbits}, thereby concluding the proof.
\begin{lemma}\label{L: spectral radius estimate}
 There exists a constant $c>0$ that does not depend on $t$ such that 
\begin{equation}\label{claim}\mathbb{P}(t_n=t)\leq e^{-cn}.\end{equation}
\end{lemma}

To prove the lemma, recall first that the length $|w_n|$ is essentially a drifted random walk on $\N$ with drift $1/3$ and it behavior is completely understood. By classical large deviations, there is a constant $c_1>0$ such that 
\[\mathbb{P}(|w_n|\leq n/4)\leq e^{-c_1n}.\]
Hence\begin{align*}\mathbb{P}(t_n=t)\ =&\ \mathbb{P}(t_n=t, |w_n|\leq n/4)+\mathbb{P}(t_n=t, |w_n|\geq n/4)\\
\leq &\ e^{-c_1n}+\mathbb{P}(t_n=t, |w_n|\geq n/4)\end{align*}
so it is sufficient to study the probability $\mathbb{P}(t_n=t)$ conditionally to the event $|w_n|\geq n/4$.

\begin{lemma}\label{L: independent Bernoulli}
Conditionally to $w_n$ and assuming that $|w_n|\geq n/4$, the values of $t_n$ on the points from the construction $t_n(x_0(w_n)), t_n(x_1(w_n)),\cdots t_n(x_{|w_n|}(w_n))$ are independent uniform Bernoulli in $\Z/2\Z$.
\end{lemma}

This concludes the proof of Lemma~\ref{L: spectral radius estimate}, since then (conditionally to $w_n$) the probability that these values are equal to  the corresponding values of $t$ is bounded above by $2^{-|w_n|}\leq 2^{-n/4}$ and by disintegrating over possible values of $w_n$ we get
\[\mathbb{P}(t_n=t, |w_n|\geq n/4)\leq \sum_{|w|\geq n/4}\mathbb{P}(w_n=w)2^{-n/4}\leq 2^{-n/4}.\]
\begin{proof}[Proof of Lemma~\ref{L: independent Bernoulli}]
Fix $i$ and consider  the sequence of vertices $(h_r\cdots h_n x_i(w_n))_{r\leq n}$. For $i\neq j$ these sequences are obtained one from each other by a suitable translation (they  stay in a region where the coloring look the same, using that $n\geq 4|w_n|$ and that the auxiliary words $\tilde{w}$ from the construction have length $\geq 5|w_n|$). In particular it is not possible that for $i\neq j$ and for the same $0\leq r\leq n$ we have $h_r\cdots h_n x_i(w_n)=x_i(w_n)+i$ and $h_r\cdots h_n x_j(w_n)=x_j(w_n)+j$. For each $i=0,\ldots , |w_n|$ let $\Sigma_i$ be the set of times $r\leq n$ so that $h_r\cdots h_n x_i(w_n)=x_i(w_n)+i$. Then  the following two properties are satisfied
\begin{enumerate}
\item $\Sigma_i\cap \Sigma_j=\varnothing$ whenever $i\neq j$ (because of the above observation).
\item $\Sigma_i\neq \varnothing$ for every $i$ (since by construction $w_n=h_n\cdots h_1$ is the word read on the right of $ x_i(w_n)$ and hence $h_1\cdots h_n x_i(w_n)=x_i(w_n)+|w_n|$.)
\end{enumerate} 
Let $\xi_0, \ldots, \xi_n$ be independent Bernoulli random variables taking values in $\Z/2\Z$. Then it is readily checked that for every $i$ the value of $t_n(x_i(w_n))$ has the same distribution as $t(h_1\cdots h_n x_i(w_n))+\sum_{r\in \Sigma_i} \xi_r$. These are all Bernoulli random variables because of the second property above, and they are independent because of the first property. \qedhere

\end{proof}

We conclude with a question; a negative answer to it would trivialize Proposition~\ref{P: counterexample unrestricted}.

 \begin{question}
 Does there exist a faithful action of a nonamenable group $G$ on a set $X$ such that the action of $(\Z/2\Z)^X\rtimes G$ on $(\Z/2\Z)^X$ is amenable?
 \end{question}

\bibliography{extAmen}

\begin{thebibliography}{AAMBV13}

\bibitem[AAMBV13]{Amir-Angel-Matte-Virag}
Gideon Amir, Omer Angel, Nicol\'as Matte~Bon, and B\'alint Vir\'ag.
\newblock The {L}iouville property for groups acting on rooted trees.
\newblock 2013.
\newblock {P}reprint, arXiv:1307.5652.

\bibitem[AAV13]{Amir-Angel-Virag:linearautomata}
Gideon Amir, Omer Angel, and B{\'a}lint Vir{\'a}g.
\newblock Amenability of linear-activity automaton groups.
\newblock {\em J. Eur. Math. Soc. (JEMS)}, 15(3):705--730, 2013.

\bibitem[AV12]{Amir-Virag:speedexponent}
Gideon Amir and B\'alint Vir\'ag.
\newblock Speed exponents for random walks on groups.
\newblock 2012.
\newblock {P}reprint, arXiv:1203.6226.

\bibitem[AV14]{Amir-Virag:positivespeed}
Gideon Amir and B{\'a}lint Vir{\'a}g.
\newblock Positive speed for high-degree automaton groups.
\newblock {\em Groups Geom. Dyn.}, 8(1):23--38, 2014.

\bibitem[BE11]{Bartholdi-Erschler:poissongrowth}
Laurent Bartholdi and Anna Erschler.
\newblock Poisson-furstenberg boundary and growth of groups.
\newblock 2011.
\newblock {P}reprint, arXiv:1107.5499.

\bibitem[BE12]{Bartholdi-Erschler1}
Laurent Bartholdi and Anna Erschler.
\newblock Growth of permutational extensions.
\newblock {\em Invent. Math.}, 189(2):431--455, 2012.

\bibitem[BKN10]{Bartholdi-Kaimanovich-Nekrashevych:boundedautomata}
Laurent Bartholdi, Vadim~A. Kaimanovich, and Volodymyr~V. Nekrashevych.
\newblock On amenability of automata groups.
\newblock {\em Duke Math. J.}, 154(3):575--598, 2010.

\bibitem[BLP77]{Baldi-Lohoue-Peryiere}
Paolo Baldi, No{\"e}l Lohou{\'e}, and Jacques Peyri{\`e}re.
\newblock Sur la classification des groupes r\'ecurrents.
\newblock {\em C. R. Acad. Sci. Paris S\'er. A-B}, 285(16):A1103--A1104, 1977.

\bibitem[Bon12]{Bondarenko}
Ievgen~V. Bondarenko.
\newblock Growth of {S}chreier graphs of automaton groups.
\newblock {\em Math. Ann.}, 354(2):765--785, 2012.

\bibitem[dC13]{Cornulier:Bourbaki}
Yves de~Cornulier.
\newblock {G}roupes pleins-topologiques [d'apr\`es {M}atui, {J}uschenko,
  {M}onod,...].
\newblock 2013.
\newblock {W}ritten exposition of the {B}ourbaki {S}eminar of {J}anuary 19th,
  2013. {A}vailable at\:{\tt www.normalesup.org/\string~cornulier/}.

\bibitem[DFG13]{DFG}
Fran{\c{c}}ois Dahmani, Koji Fujiwara, and Vincent Guirardel.
\newblock Free groups of interval exchange transformations are rare.
\newblock {\em Groups Geom. Dyn.}, 7(4):883--910, 2013.

\bibitem[EM13]{Elek-Monod}
G{\'a}bor Elek and Nicolas Monod.
\newblock On the topological full group of a minimal {C}antor
  {$\bold{Z}^2$}-system.
\newblock {\em Proc. Amer. Math. Soc.}, 141(10):3549--3552, 2013.

\bibitem[JdlS13]{JS}
Kate Juschenko and Mikael de~la Salle.
\newblock {I}nvariant means of the wobbling group.
\newblock 2013.
\newblock {P}reprint, arXiv 1301.4736.

\bibitem[JM13]{JM}
Kate Juschenko and Nicolas Monod.
\newblock Cantor systems, piecewise translations and simple amenable groups.
\newblock {\em Ann. of Math. (2)}, 178(2):775--787, 2013.

\bibitem[JNdlS13]{JNS}
Kate Juschenko, Volodymyr Nekrashevych, and Mikael de~la Salle.
\newblock Extensions of amenable groups by recurrent groupoids.
\newblock 2013.
\newblock Preprint, arXiv:1305.2637v2.

\bibitem[Kea75]{keane}
Michael Keane.
\newblock Interval exchange transformations.
\newblock {\em Math. Z.}, 141:25--31, 1975.

\bibitem[KH95]{KH}
Anatole Katok and Boris Hasselblatt.
\newblock {\em Introduction to the modern theory of dynamical systems},
  volume~54 of {\em Encyclopedia of Mathematics and its Applications}.
\newblock Cambridge University Press, Cambridge, 1995.
\newblock With a supplementary chapter by Katok and Leonardo Mendoza.

\bibitem[KS67]{KS}
A.~B. Katok and A.~M. Stepin.
\newblock Approximations in ergodic theory.
\newblock {\em Uspehi Mat. Nauk}, 22(5 (137)):81--106, 1967.

\bibitem[KV83]{Kaimanovich-Vershik}
V.~A. Ka{\u\i}manovich and A.~M. Vershik.
\newblock Random walks on discrete groups: boundary and entropy.
\newblock {\em Ann. Probab.}, 11(3):457--490, 1983.

\bibitem[LP14]{LyonsPeres}
Russell Lyons and Yuval Peres.
\newblock {\em Probability on Trees and Networks}.
\newblock Cambridge University Press, 2014.
\newblock book in preparation, December 2014 version. Available at
  http://mypage.iu.edu/\string~rdlyons/prbtree/prbtree.html.

\bibitem[MB14]{Matte}
Nicol{\'a}s Matte~Bon.
\newblock Subshifts with slow complexity and simple groups with the {L}iouville
  property.
\newblock {\em Geom. Funct. Anal.}, 24(5):1637--1659, 2014.

\bibitem[Mon13]{FM}
Nicolas Monod.
\newblock Groups of piecewise projective homeomorphisms.
\newblock {\em Proc. Natl. Acad. Sci. USA}, 110(12):4524--4527, 2013.

\bibitem[MP03]{Monod-Popa}
Nicolas Monod and Sorin Popa.
\newblock On co-amenability for groups and von {N}eumann algebras.
\newblock {\em C. R. Math. Acad. Sci. Soc. R. Can.}, 25(3):82--87, 2003.

\bibitem[Sid00]{Sidki}
Said Sidki.
\newblock Automorphisms of one-rooted trees: growth, circuit structure, and
  acyclicity.
\newblock {\em J. Math. Sci. (New York)}, 100(1):1925--1943, 2000.
\newblock Algebra, 12.

\bibitem[vD90]{VanDouwen}
Eric~K. van Douwen.
\newblock Measures invariant under actions of {$F_2$}.
\newblock {\em Topology Appl.}, 34(1):53--68, 1990.

\bibitem[Via06]{viana}
Marcelo Viana.
\newblock Ergodic theory of interval exchange maps.
\newblock {\em Rev. Mat. Complut.}, 19(1):7--100, 2006.

\bibitem[Woe00]{Woess}
Wolfgang Woess.
\newblock {\em Random walks on infinite graphs and groups}, volume 138 of {\em
  Cambridge Tracts in Mathematics}.
\newblock Cambridge University Press, Cambridge, 2000.

\end{thebibliography}
\bibliographystyle{alpha}

\end{document}